\documentclass[12pt]{article}
\usepackage{amssymb,mathrsfs,amsmath}
\usepackage{amscd,amsmath,amsthm,amssymb,amsfonts,enumerate}
\usepackage{color}
\def\P{{\mathbb P}}
\def\E{{\mathbb E}}

\newtheorem{thm}{Theorem}[section]
\newtheorem{prop}[thm]{Proposition}
\newtheorem{lem}[thm]{Lemma}
\newtheorem{rem}[thm]{Remark}

\newtheorem{defn}[thm]{Definition}
\newtheorem{exm}[thm]{Example}

\def\ga{\gamma}

\newcommand{\thmref}[1]{Theorem~{\rm \ref{#1}}}
\newcommand{\lemref}[1]{Lemma~{\rm \ref{#1}}}

\makeatletter
\@addtoreset{equation}{section}

\newcommand{\beq}[1]{\begin{equation} \label{#1}}
\newcommand{\eeq}{\end{equation}}
\newcommand{\bed}{\begin{displaymath}}
\newcommand{\eed}{\end{displaymath}}
\newcommand{\disp}{\displaystyle}

\newcommand{\bedd}{\bed\begin{array}{l}}
\newcommand{\eedd}{\end{array}\eed}
\newcommand{\al}{\alpha}
\newcommand{\sg}{\sigma}
\newcommand{\be}{\beta}
\newcommand{\e}{\varepsilon}

\newcommand{\dl}{\delta}
\newcommand{\nd}{\noindent}

\newcommand{\La}{\Lambda}

\newcommand{\M}{{\cal{M}}}

\def\indi{{\bf 1}}
\def\K{{\mathcal K}}

\newcommand{\LL}{{\cal L}}
\newcommand{\wt}{\widetilde}
\newcommand{\RR}{{\mathbb R}}

\newcommand{\wdt}{\widetilde}
\newcommand{\wdh}{\widehat}

\newcommand{\cd}{(\cdot)}
\newcommand{\rr}{\Bbb R}
\newcommand{\lbar}{\overline}

\def\sgn{\hbox{sgn}}

\def\ka{\kappa}

\def\G{{\cal G}}
\def\L{{\cal L}}

\def\({\left(}
\def\){\right)}
\def\one{{\hbox{1{\kern -0.35em}1}}}
\def\R{{\mathbb R}}

\newcommand{\bea}{\bed\begin{array}{rl}}
\newcommand{\eea}{\end{array}\eed}
\newcommand{\ad}{&\!\!\!\disp}
\newcommand{\aad}{&\disp}
\newcommand{\barray}{\begin{array}{ll}}
\newcommand{\earray}{\end{array}}

\def\tr{\hbox{tr}}
\newcommand{\sumj}{\sum^{m}_{j=1}}
\newcommand{\sumi}{\sum^{m}_{i=1}}

\begin{document}

\title{
Recurrence and Ergodicity for
A Class of Regime-Switching Jump Diffusions}

\author{Xiaoshan Chen,\thanks{School of Mathematical Sciences, South China Normal University, Guangdong, China,
xschen@m.scnu.edu.cn. This research was partially supported by NNSF of China (No.11601163) and NSF of Guangdong Province of China (No. 2016A030313448).} \and Zhen-Qing Chen,\thanks{Departments of Mathematics,
University of Washington,
Seattle, WA 98195, USA, zqchen@uw.edu. This
research was partially supported
by NSF grant DMS-1206276.
} \and Ky Tran,\thanks{Department
of Mathematics, College of Education, Hue University, Hue city, Vietnam.
quankysp@gmail.com. This research was
partially supported by Vietnam National Foundation for Science and Technology Development (NAFOSTED) under grant 101.03-2017.23.}
\and George Yin\thanks{Department
of Mathematics, Wayne State University, Detroit, MI 48202 USA,
gyin@math.wayne.edu. This research was
partially supported by the National Science Foundation under grant DMS-1710827.}
}

\date{}

\maketitle

\begin{abstract}
This work develops asymptotic properties of a class of switching jump diffusion processes.
The processes under consideration may be viewed as a number of jump diffusion processes modulated by
a random switching mechanism.
The underlying processes feature in the switching process depends on the jump diffusions.
In this paper, conditions for recurrence and positive recurrence are derived. Ergodicity is examined in detail.
Existence of invariant probability measures is proved.

\vskip 0.3 true in \nd{\bf Key words.} Switching jump diffusion, regularity, recurrence, positive recurrence, ergodicity.

\vskip 0.3 true in \nd{\bf Mathematics subject classification.} 60J60, 60J75, 60G51.

\vskip 0.3 true in \nd{\bf Brief Title.} Recurrence and Ergodicity of Switching Jump Diffusions

\end{abstract}

\setlength{\baselineskip}{0.24in}

\newpage

\section{Introduction}\setcounter{equation}{0}

This work is concerned with the asymptotic properties,
in particular recurrence and ergodicity, for
 a class of regime-switching
jump diffusion processes. The process under consideration can be thought of as
a number of jump-diffusion processes modulated by a random switching device. The
underlying process is  $\(X(t), \La(t)\)$, a two-component process, where $X(t)$ delineates the
jump diffusion behavior and $\La(t)$ describes the switching involved. One of the
main ingredients is that the switching process depends on the jump diffusions.
If the jump process is missing, it reduces to switching diffusions that  have
been investigated thoroughly; see for example \cite{MY06,YZ10} and references therein.
Our primary motivation stems from the study of  a family of Markov processes in
which continuous dynamics, jump discontinuity, and discrete events coexist.
Their interactions reflect the salient features of the underlying systems.
 The distinct features of the systems include the presence of non-local operators, the coupled systems
of equations, and the tangled information due to the dependence of the switching process on the jump diffusions.
Such systems have drawn new as well as resurgent attention because of the urgent needs of systems modeling, analysis,
and optimization in a wide variety of applications. Not only do the applications arise from the traditional fields of mathematical modeling,
but
also they have appeared in emerging application areas such as  wireless communications, networked systems,
autonomous systems, multi-agent systems, flexible manufacturing systems,
financial engineering, and
biological and
ecological systems, among others.
Continuing our investigation on regime-switching systems,
this paper focuses on
a class of switching jump diffusion processes. The
state-dependent
random switching
process under consideration makes the formulation more versatile and interesting with a wider range of
applications. Nevertheless, it makes the analysis more difficult and challenging.

Asymptotic properties of diffusion processes and  associated
partial differential equations are well known in the
literature. We refer to \cite{B78, K80} and references therein.
Results for switching diffusion processes can be found
in \cite{GAP, ZY07, YZ10}.
One of the important problems concerning switching jump models is
their longtime behavior. In particular, similar to the case of diffusions and switching diffusion processes, several questions of
 crucial importance are: Under what conditions will the systems return to a prescribed compact region in finite time? Under what conditions do the systems have the desired ergodicity? In this paper, we focus on asymptotic behavior and address these issues.
Despite the growing interests in treating
switching jump systems, the results regarding
recurrence, positive recurrence, and invariant probability measures are still scarce.
One of the main difficulties is the operator being non-local.  When we study switching diffusions, it has been demonstrated that
although they are similar to
diffusion processes, switching diffusions have some  distinct features.
With the non-local operator used, the distinctions are even more pronounced.

In the literature, criteria for certain types of weak stability including Harris recurrence and positive Harris recurrence for continuous time Markovian processes based on
Foster-Lyapunov inequalities
were developed in \cite{MT93}. Using results in that paper, some sufficient conditions for ergodicity of L\'evy type operators in dimension one were presented in \cite{W08} under the assumption of Lebesgue-irreducibility. In \cite{W99, YX10},  sufficient conditions for stability of a class of jump diffusions and switching jump diffusions
 were provided
  by constructing suitable
Lyapunov
type functions. However, the class of kernels considered in aforementioned papers satisfies a different set of hypotheses than those in our current work.
 In \cite{M99}, the existence of an invariant probability measure was established for a jump diffusion, whose drift coefficient is assumed to be only Borel measurable.
Sufficient conditions for recurrence, transience and ergodicity of stable-like processes
were studied in \cite{S13}.
 In a recent work \cite{ABC16}, the authors treated the ergodic properties such as positive recurrence and invariant probability measures for jump processes with no diffusion part. It should be noted that for the case of switching diffusions, the state process $X(t)$ can be
 viewed as a diffusion process
 in a random environment characterized by the switching component $\La(t)$. The asymptotic behavior of the diffusion process
 in a random environment is more complicated than that of a diffusion process in a
 fixed environment. In particular, there are examples (see \cite{PP93, PS92}) that $X(t)$ is positive recurrent in some fixed environments and is transient in some other fixed environments, we can make $X(t)$ positive recurrent or transient by choosing suitable switching rates. We refer to \cite{SX13, SX14} for recent works on ergodicity of regime-switching diffusion processes. We emphasize that our results in the current work are significant  extensions of those in \cite{ABC16, ZY07}.
 In terms of the associated operators and differential equations,
  compared to switching diffusion processes, in lieu of an
elliptic system (a system of elliptic partial differential equations), we have to deal with a system of
integro-differential equations. In our recent work,
 maximum principle and Harnack's inequality were obtained in  \cite{CC16}, which are used
 in this paper. Compared to the
case of diffusion processes with jumps, even though the classical
approaches such as
Lyapunov
function methods and Dynkin's formula
are still applicable, the analysis is much more delicate because of
the coupling and interactions.
There is a wide range of applications. As one
 particular example, consider the following average cost per unit time problem for a controlled switching jump diffusion
 in an infinite horizon.
 Suppose that $(X(t),\Lambda(t))$ is given by
  \bea \ad dX(t) = b(X(t), \Lambda(t), u(t)) dt +\sigma (X(t), \Lambda(t)) dW(t) +  \int _{\RR_0} c (X({t-}),\Lambda({t-}),z) \wdt N_0(dt, dz), \\ \ad  X(0)=x, \ \Lambda(0)= \lambda_0, \eea
 where $b\cd$, $\sigma\cd$, and $c \cd$ are suitable real-valued functions, $\wdt N_0\cd$ is a compensated real-valued Poisson process,
 $W\cd$ is a real-valued Brownian motion, $\Lambda(t)$ is a continuous-time Markov chain with a finite state space,
 $u\cd$ is a control,
 and $\RR_0= \RR - \{0\}$.
 [More precise notion of switching jump diffusions will be given in the next section.]
 Assuming that $\Lambda\cd$, $W\cd$, and $N_0\cd$ are independent,
our objective is to minimize a long-run average cost function given by
$$J(x, \lambda_0, u\cd) = \lim_{T\to \infty} \E {1\over T} \int^T_0 g(X(t),\Lambda(t), u(t)) dt,
$$ where $g\cd$ is a suitable ``running cost'' function.
To treat such a problem, one needs to replace the instantaneous measure by an ergodic measure (if it exists). The current paper sets up the foundation for the study on such problems because it provides sufficient conditions under which the invariant measure exists.
Optimal controls of controlled jump diffusions (the above process without switching) was considered in \cite{Ku90}, in which the jump process was assume to belong to a compact set. Here we are dealing with a more general setting.

The rest of the paper is arranged as follows. Section \ref{sec:for}
begins with the formulation of the problem. Section \ref{sec:rec}
is devoted to recurrence. It  provides
definitions of  recurrence, positive recurrence, and null
recurrence in addition to introducing certain notation. Then
Section \ref{sec:pos} focuses on recurrence and positive recurrence.
We present  sufficient conditions
involving Lyapunov function for
recurrence and positive recurrence using
Lyapunov
functions. Section \ref{sec:inva} develops
ergodicity. Existence of
invariant probability measures of switching jump diffusion processes is obtained.
Section \ref{sec:exm}
provides some sufficient conditions for the existence of Lyapunov functions under which the theorems
of this paper are applicable.
Finally, Section \ref{sec:rem} concludes the paper with further remarks.

\section{Formulation} \label{sec:for}

Throughout the paper,
we use $z'$ to denote the transpose of $z\in \R^{l_1
	\times l_2}$ with $l_1, l_2 \ge 1$, and $\R^{d\times 1}$ is simply
written as $\R^d$. If $x\in \R^d$, the norm of $x$ is denoted by
$|x|$.
For $x\in \R^d$ and $r>0$,
we use $B(x, r)$ to denote the open ball in $\R^d$ with radius $r$ centered at $x$.
The term {\sl domain} in $\R^d$ refers to a nonempty connected open
subset of the Euclidean space $\R^d$.
If $D$ is a domain in $\R^d$, then $\overline{D}$ is the closure
of $D$, $D^c=\R^d \setminus D$ is its complement. The space $C^2(D)$ refers to the class of functions whose partial derivatives up to order 2 exist and are continuous in $D$, and $C^2_b(D)$ is the subspace of $C^2(D)$ consisting of those functions whose partial derivatives up to order 2 are bounded. The indicator function of a set $A$ is denoted by $\indi_A$.

Let $(X(t),\Lambda(t))$ be a two component Markov process such that $X(\cdot)$ is an $\R^d$-valued process, and $\Lambda(\cdot)$ is a switching process taking values in a finite set $\mathcal M=\{1,2,\dots,m\}$.
Let $b(\cdot,\cdot): \R^d\times \mathcal M\mapsto \R^d$, $\sigma(\cdot,\cdot): \R^d\times\mathcal M\mapsto\R^d\times\R^d$, and
for each $x\in \R^d$ and $i\in \M$, $\pi_i (x, dz)$ is a $\sigma$-finite measure on $\R^d$  satisfying
$$
\int_{\R^d} (1\wedge |z|^2) \pi_i (x, dz) <\infty.
$$
Let $Q(x)=(q_{ij}(x))$ be an $m\times m$ matrix depending on $x$ such that $$q_{ij}(x)\geq 0 \quad \text{for } i\ne j, \quad  \sum_{j\in\mathcal{M}}q_{ij}(x)= 0.$$ Define
$$Q(x)f(x,\cdot)(i):=\sum_{j\in\mathcal{M}}q_{ij}(x)f(x,j).$$
The generator $\mathcal G$ of the process $(X(t),\Lambda(t))$ is given as follows. For a function $f:\RR^d\times \M\mapsto \RR$ and $f(\cdot, i)\in C^2(\R^d)$
for each $i\in\mathcal M$, define
\begin{equation}\label{E:1}
\mathcal Gf(x,i)
= \LL_i f(x,i)+ Q(x)f(x, \cdot)(i), \quad   (x, i)\in \R^d\times \M,\end{equation}
where
\begin{eqnarray}
\mathcal{L}_i f(x,i) &=&\dfrac{1}{2}\sum^d_{k,l=1}a_{kl}(x,i){\frac{\partial^2f(x,i)}{\partial x_k\partial x_l}+\sum^d_{k=1}b_{k}(x,i)\frac{\partial f(x,i)}{\partial x_k}}
\nonumber\\
&& +\int_{\R^d}  \left( f(x+z,i)-f(x,i)-\nabla f(x,i)\cdot z \indi_{\{|z|<1\}}\right)
\pi_i(x, dz) ,    \label{e:2.2}
\end{eqnarray}
$a(x,i):=\big(a_{kl}(x, i)\big)=\sigma(x,i)\sigma'(x,i)$ and $\nabla f(\cdot,i)$
denotes the gradient  of $f(\cdot,i)$.

Let $\Omega=D\big([ 0, \infty ), \R^d\times \M \big)$
 be the space of functions
 (mapping $[0, \infty)$ to $\R^d \times \M$)
 that are right continuous with left limits endowed with
  the Skorohod topology.  Define $(X(t), \Lambda(t))=w(t)$ for  $w\in \Omega$
and let $\{\mathcal{F}_t\}$ be the right continuous filtration generated by the process $(X(t), \Lambda(t))$.
A probability measure $\P_{x,i }$ on $\Omega$ is a solution to the martingale problem for $\(\mathcal{G}, C^2_b(\R^d)\)$ started at $(x, i)$ if
\begin{itemize}
	\item[(a)] $\P_{x, i}( X(0)=x, \Lambda(0)=i )=1$,
	\item[(b)] if $f(\cdot, i)\in C^2_b(\R^d)$ for each $i\in \M$, then
	$$f(X(t), \Lambda(t))-f(X(0), \Lambda(0))-\int_0^t \mathcal{G} f( X(s), \Lambda(s) )ds,$$
	is a $\P_{x, i}$ martingale.
	
	If for each $(x, i)$, there is only one such $\P_{x, i}$, we say that the martingale problem for
	$\(\mathcal{G}, C^2_b(\R^d)\)$ is well-posed.
\end{itemize}

	Throughout the paper,   we assume conditions (A1)-(A4) hold
	unless otherwise noticed.

	\begin{itemize}
		\item[{(A1)}] The functions  $\sigma
		(\cdot, i)$ and $b (\cdot, i)$ are   continuous on $\R^d$ for each $i\in \M$ such that
		$$ \sum_{k,l=1}^d |\sigma_{kl}(x, i)| + \sum_{k=1}^d |b_k(x, i)| \leq c (1+|x|) \quad \hbox{on } \R^d
		$$
		for some $c>0$,
		and $q_{ij}\cd$ is bounded and Borel measurable on $\R^d$ for $i,j\in \M$.

	\item[{(A2)}] For every bounded domain $D\subset \R^d$, there exists a constant $\kappa_0=\kappa_0 (D) \in (0, 1]$ such that
		$$ \kappa_0|\xi|^2\le \xi'a(x, i)\xi \le
		\kappa_0^{-1}|\xi|^2 \quad  \text{ for all }\xi\in \R^d, x\in D,  \, i\in \M.
		$$
		
		\item[{(A3)}] We have 
  $\sup_{x\in \R^d, i\in \M} \pi_i (x, B(0, 1)^c)<\infty$ and $\sup_{i\in \M} \int_{B(0, 1)} |z|^2 \pi_i (x, dz) \leq c(1+|x|^2)$ for some $c>0$.
  Moreover, for every $n\geq 1$,    there exists a $\sigma$-finite measure $\Pi _n(dz)$ with
                           $\int_{\R^d} \left(1\wedge |z|^2\right) \Pi_n(dz)  <\infty$   so that   $\pi_i (x, dz) \leq \Pi_n (dz)$
 		for every $x\in B(0, n)$ and $i\in \M$.

		\item [(A4)]
		For any $i\in \M$ and $x\in \R^d$, $\pi_i (x, dz)=\wdt  \pi_i (x, z) dz$. Moreover,  for
		any $r\in (0,1)$, any $x_0\in \R^d$, any $x, y \in B(x_0, r/2)$ and $z\in B(x_0, r)^c$, we have
		$$\wdt\pi_i(x,z-x)\leq \al_r\wdt\pi_i(y,z-y),$$ where $\alpha_r$ satisfies $1<\alpha_r<\kappa_2 r^{-\beta}$ with $\kappa_2$ and $\beta$ being positive constants.		
	\end{itemize}

	\begin{rem}\label{R:2.2} {\rm We comment on the conditions briefly.
\begin{itemize}
			\item[(a)] Under Assumptions (A1)-(A3), for each $i\in \M$,
                                 there is a unique solution to the martingale problem for $(\LL_i, C^2_b(\R^d))$ which is
                                 conservative
                                 (see \cite[Theorem 5.2]{Komatsu}).
			Note that the switched Markov process $\(X(t), \Lambda(t)\)$ can be constructed from that of  $\LL_i$
			as follows.
			Let $X^i$ be the Markov process associated with $\LL_i$, that is, $\LL_i$ is the infinitesimal generator of $X^i$.
			Suppose we start the process at $(x_0, i_0)$, run a subprocess $\wt X^{i_0}$ of $X^{i_0}$ that got killed with rate
			$-q_{i_0i_0} (x)$; that is, via Feynman-Kac transform $\exp \left(\int_0^t q_{i_0i_0} (X^{i_0}(s))ds \right)$. Note that this subprocess $\wt X^{i_0}$ has infinitesimal generator  $\LL_{i_0}+q_{i_0i_0}$.
			At the lifetime $\tau_1$ of the killed process $\wt X^{i_0}$, jump to
			plane $j\ne i_0$
			with probability $-q_{i_0j} (X^{i_0}(\tau_1-))/q_{i_0i_0} (X^{i_0}(\tau_1-))$
			and run a subprocess $\wt X^j$ of $X^j$ with killing rate $-q_{jj}(x)$ from position $X^{i_0}(\tau_1-)$. Repeat this procedure. As pointed out in \cite{CC16},	since $\sum_{j\in \M}q_{ij}(x)=0$ on
			$\R^d$ for every $i\in \M$,	the resulting process $\(X(t), \Lambda(t)\)$
			is a strong Markov process with
			lifetime $\zeta=\infty$. It is easy to check that the law of $\(X(t), \Lambda(t)\)$
			solves the martingale problem for $\(\mathcal{G}, C^2_b(\rr^d)\)$ so  it is
			the desired switched jump-diffusion.   It follows from \cite{W14} that  the  law of $\(X(t), \Lambda(t)\)$
			is the unique solution to the martingale problem for $\(\mathcal{G}, C^2_b(\rr^d)\)$.

			\item[(b)]	  Conditions (A1) and (A2) present
			the local uniform ellipticity of $a(x, i)$ and the linear growth bound on $a(x, i)^{1/2}$ and $|b(x, i)|$.
			The measure $\pi_i(x, dz)$
			can be thought of as the intensity of the number of jumps from $x$ to $x+z$ (see \cite{BK05, BL02}).
			Condition
			(A4) tells us that $\pi_i(x, dy)$ is absolutely continuous with respect to the Lebesgue measure $dx$ on $\rr^d$, and
			the intensities of jumps from $x$ and $y$ to a point $z$ are comparable if
			$x$, $y$ are relatively far from $z$ but relatively close to each other. If $\wdt\pi_i(x, z)$  is such that
			$$
			\frac{c_i^{-1}}{|z|^{d+\alpha_i}} \leq \wdt\pi_i (x, z) \leq \frac{c_i }{|z|^{d+\alpha_i}}
			$$
			for some $c_i\geq 1$ and $\alpha_i\in (0, 2)$,  then  condition (A4) is satisfied
			with $1< \al_r< \kappa_2$ independent of $r\in (0, 1)$.
			Condition (A4) is needed for Harnack inequality for the switched jump-diffusion; see \cite{CC16}.
			
			\item[(c)] In our recent paper \cite{CC16}, various properties of switching jump diffusions including maximum principle and Harnack inequality
			   are studied under the assumption that $a(x, i)$ is uniformly  elliptic and $b(x, i)$ is bounded on $\R^d$ and that  there is one single $\Pi$
			   so that (A3) holds with $\Pi_n=\Pi$ for all $n\geq 1$. Using localization, one can easily see that the results of \cite{CC16} hold under the current setting of this paper when they are applied to bounded open sets. In particular, for each bounded open subset $D\subset \R^d$, Theorems 3.4 and 4.7 of \cite{CC16}
			  as well as Proposition 4.1 (for all $x_0\in D$) and Proposition 4.3  hold under the setting of this paper.
	\end{itemize}	}\end{rem}

Recall that a regime-switching jump diffusion $\(X(t), \La(t)\)$ can be regarded as the results of $m$ jump diffusion processes $X^1(t), X^2(t) \dots, X^m(t)$  switching from one to another according to the dynamic movement of $\La(t)$, where $X^i(t)$ is the Markov process associated with $\L_i$. By assumption (A4), if $\wdt \pi_i(x, z)=0$ for some $(x, z)\in \RR^d\times \RR^d
$, then $\wdt \pi_i(x, z)\equiv 0$ on $\RR^d\times \RR^d
$; that is, $X^i(t)$ is a diffusion process. If $\wdt \pi_i(x, z)\equiv 0$ on $\RR^d\times \RR^d
$ for all $i\in \M$, $\(X(t), \La(t)\)$ is just a switching diffusion process. Thus, the class of models under consideration includes switching
diffusions.
In this paper, sufficient conditions for recurrence, positive recurrence, and the existence of an invariant probability measure are given. We show
that under the sufficient conditions derived, recurrence and positive recurrence are independent of the domain chosen.
Furthermore, it is demonstrated that we can work with a fixed discrete component.
Several examples are provided in which  easily verifiable conditions on coefficients of the switching jump diffusions are given in lieu of the
Lyapunov
function type of conditions.

\section{Recurrence}\label{sec:rec}
This section is devoted to
recurrence, positive recurrence, and null recurrence.
The first part gives definitions and the second part
provides criteria for recurrence and positive recurrence.

\subsection{Definitions: Recurrence, Positive Recurrence, and Null
Recurrence}
 For
simplicity, we introduce some notation as follows. For any $U
=D\times J\subset \rr^d\times \M$, where $D\subset \rr^d$ and
$J\subset \M$. Define \bed\barray \aad \tau_U:=\inf\{t\ge 0: (X(t)),
\La(t))\notin U\}, \\
\aad \sg_U:=\inf\{t\ge 0: (X(t)), \La(t))\in
U\}. \earray\eed
In particular, if $U=D\times \M$ is a ``cylinder",
we set \bed\barray \aad \tau_D:=\inf\{t\ge 0: X(t)\notin D\},
\\\aad \sg_D:=\inf\{t\ge 0: X(t)\in D\}. \earray\eed

\begin{rem}{\rm Recall that	the  process $\(X(t), \Lambda(t)\)$
		is a strong Markov process with
		lifetime $\zeta=\infty$. Let $\be_n=\inf\{t\ge 0: |X(t)|\ge n\}$ be the first exit time of
the process $\(X(t)), \La(t)\)$ from the bounded set $B(0, n)\times \M$. Then the sequence $\{\be_n\}$ is monotonically
increasing
and $\be_n \to \infty $ almost surely
as $n\to \infty$.
We also refer to such property of  $\(X(t), \Lambda(t)\)$
as regularity or
 no finite explosion
time.
}\end{rem}

Recurrence, positive recurrence, and null recurrence are defined
similarly as for switching diffusions given in \cite{ZY07, YZ10}.

\begin{defn}\label{def:def}{\rm
Suppose $U=D\times J$, where $J\subset \M$ and $D\subset \rr^d$ is a
bounded domain.
A  process $(X(t),
\La(t))$ is said to be \emph{recurrent} with respect to $U$
if $$\P_{x, i}\(\sg_U<\infty\)=1 \quad \text{for any} \quad (x, i)\in
U^c .
$$
If $(X(t),
\La(t))$ is recurrent with respect to $U$ and $\E_{x,\, i}\sg_U<\infty$ for any $(x, i)\in
U^c$,
 then it said to be \textit{positive recurrent} with respect to $U$; otherwise, the process is \emph{null recurrent}
with respect to $U$.}
\end{defn}

\subsection{Criteria for Recurrence and Positive Recurrence} \label{sec:pos}
We begin this section with
some
preparatory results, followed by sufficient conditions for  recurrence and positive recurrence of the process
$(X(t), \La(t))$ with respect to some cylinder $U=D\times \M$.
Under these sufficient conditions, we prove that the process $(X(t), \La(t))$ is recurrent
(resp., positive recurrent) with respect to some cylinder $D\times
\M$ if and only if it is recurrent (resp., positive recurrent) with
respect to $D\times \{l\}$
for every $l\in \M$, where $D\subset \rr^d$ is a bounded domain.
We will also prove that the
properties of recurrence and positive recurrence do not depend on
the choice of the domain $D\subset \rr^d$ and $l\in \M$.
We first prove the following theorem, which asserts
that the process $(X(t),\La(t))$ will
exit every bounded cylinder with a finite mean exit time.

\begin{thm}\label{thm:exit} Let $D\subset \rr^d$ be a bounded domain. Then \beq{e:7}
\sup\limits_{(x, i)\in D\times \M}\E_{x, i}\tau_D<\infty.\eeq
\end{thm}

\begin{proof}
By the local uniform ellipticity condition in (A2), there exists some $\kappa_0\in (0,1]$ such that
\beq{e:8} \kappa_0 \le a_{11}(x, i)\le \kappa_0^{-1}\text{
for any } (x, i)\in D\times \M.\eeq
Let $f\in C^2_b(\rr^d)$ be such that $f$ is nonnegative and  $f(x)=(x_1+\be)^\ga$ if $x\in \{y: d(y, D)<1\}$,
where the constants
$\ga >2$ and $\be >0$
are to be specified, and $x_1 = e_1'x$ is the first
component of $x$, with $e_1 = (1, 0, . . . , 0)'$ being the standard unit vector.
Define $V(x, i)=f(x)$ for $(x, i)\in \rr^d \times \M.$
Since $D$ is  bounded, we can choose constant
$\be >0$ such that
$$
1\le x_1+\be \quad\text{for all } (x, i)\in D\times \M,
$$
and
$$\ga:=\dfrac{2}{\kappa_0}\(\sup\limits_{(x, i)\in D\times \M}\Big[|b_1(x, i)|(x_1+\be) +  \kappa_1(x_1+\beta)^2\Big]+1\)+2<\infty,$$
where $\ka_1$ is the constant which exists by assumption (A3).
Then we have by \eqref{e:8} that
\beq{e:9}b_1(x,
i)(x_1+\beta)+\dfrac{\ga-1}{2}a_{11}(x, i)-{\ka_1}(x_1+\be)^2 \ge 1 \quad \text{for all}\quad (x, i)\in D\times \M.\eeq
Direct
computation  leads to \beq{e:10}\barray\G V(x,
i)=\ad \ga(x_1+\beta)^{\ga-2}\left[b_1(x,
i)(x_1+\beta)+\dfrac{\ga-1}{2}a_{11}(x, i)\right]\\
\ad +\int_{|z|\le 1} \big[f(x+z)-f(x)-\nabla f(x)\cdot z \big]\wdt \pi_i(x,z)(dz)\\
\ad +\int_{|z|> 1} \big[f(x+z)-f(x)\big]\wdt \pi_i(x,z)(dz). \earray
\eeq
Since $\ga>2$, $f(\cdot)$ is convex on $\{x\in \RR^d: d(x, D)<1\}$. It follows that
\beq{e:11}\int_{|z|\le 1} \big[f(x+z)-f(x)-\nabla f(x)\cdot z \big]\wdt \pi_i(x,z)(dz)\ge 0.\eeq
By the nonnegativity of $f$, it is also clear that
\beq{e:12}\barray \disp \int_{|z|> 1} \big[f(x+z)-f(x)\big]\wdt \pi_i(x,z)(dz) \ad \ge -\int_{|z|> 1}(x_1+\be)^\ga \wdt \pi_i(x,z)(dz)\\
\ad \ge -\ka_1 (x_1+\be)^\ga.
\earray\eeq
It follows from \eqref{e:10}, \eqref{e:11}, \eqref{e:12}, and \eqref{e:9} that
\beq{e:13}\barray
\G V(x, i)\ad \ge \ga(x_1+\beta)^{\ga-2}\left[b_1(x,
i)(x_1+\beta)+\dfrac{\ga-1}{2}a_{11}(x, i)\right]-\ka_1 (x_1+\be)^\ga\\
\ad \ge \ga(x_1+\beta)^{\ga-2}\left[b_1(x,
i)(x_1+\beta)+\dfrac{\ga-1}{2}a_{11}(x, i)-{\ka_1}(x_1+\be)^2\right]\\
\ad \ge \ga,
\earray
\eeq
 for all $(x, i)\in D\times \M$. Let
$\tau_D(t)=\min\{t, \tau_D\}$. Then we have from Dynkin's formula
and \eqref{e:13} that \bed
\barray \aad \E_{x, i}V\(X(\tau_D(t)), \La(\tau_D(t))\)-V(x, i)\\
\aad \qquad\qquad=\E_{x, i}\int_0^{\tau_D(t)}\G V(X(s), \La(s))ds\ge
\ga \E_{x, i}\tau_D(t). \earray \eed Hence  \beq{e:14} \E_{x,
i}\tau_D(t)\le \dfrac{1}{\ga}\sup\limits_{(x, i)\in \rr^d\times\M}V(x, i):=M.\eeq
Note that $M$ is finite by  our construction of functions $V\cd$.
Since $\E_{x, i}\tau_D(t)\ge
t\P_{x, i}[\tau_D>t]$,  it follows from \eqref{e:14} that
$$t\P_{x,
i}\big(\tau_D>t\big)\le M.$$ Letting $t\to \infty$, we
obtain $\P_{x, i}\big(\tau_D=\infty\big)=0.$ That is, $\P_{x, i}\big(\tau_D<\infty\big)=1.$
This yields that $\tau_D(t)\to \tau_D$ a.s. $\P_{x, i}$ as $t\to
\infty$. Now applying Fatou's lemma, as $t\to \infty$ in \eqref{e:14} we obtain \beq{e:15} \E_{x, i} \tau_D\le
M<\infty.\eeq This proves the theorem.
\end{proof}

To study
the recurrence and positive recurrence  of the process $\(X(t), \La(t)\)$, we first present criteria based on the existence of certain
Lyapunov functions.
Sufficient conditions on the existence of such Lyapunov functions will be given in Section \ref{sec:exm}.

\begin{thm}\label{thm:prec}
A sufficient condition for the positive recurrence of $\(X(t), \La(t)\)$ with
respect to $U=D\times \M\subset \rr^d\times \M$ is the following condition {\rm (C1)}
holds:
\begin{itemize}
	\item[\rm (C1)] For
	each $i\in \M$, there exists a nonnegative  function
	$V(\cdot, i)\in C^2(\rr^d)$ satisfying
	\beq{e:16}\G V(x, i)\le -1 \text{ for any }(x, i)\in
	D^c\times \M.\eeq
\end{itemize}
\end{thm}

\begin{proof}
Assume that there exists a
nonnegative function $V(\cdot, \cdot)$ satisfying condition (C1) with respect to $U=D\times \M$. We show that the process $(X(t), \La(t))$ is positive
recurrent with respect to $U=D\times \M$.

Choose $n_0$  a positive integer sufficiently large so that
$D\subset B(0, n_0)$. Fix any $(x, i)\in D^c \times \M$. For any
$t>0$ and $n\ge n_0$, we define $$\sg_D^{(n)}(t)=\min\{\sg_D, t,
\be_n\},$$ where $\be_n$ is the first exit time from $B(0, n)$ and
$\sg_D$ is the first entrance time to $D$. Let $f_n:\rr^d\mapsto \rr$ be a smooth cut-off function that takes values in $[0, 1]$
 satisfying $f_n= 1$ on $B(0, n)$ and $f_n=0$ outside of $B(0, n+1)$.
Then $V_n(\cdot, j):=f_n(\cdot) V(\cdot, j)\in C^2_b(\rr^d)$ for each $j\in \M$. Moreover,
$$0\le V_n(x, i)\le V(x, i), \qquad (x, i)\in \rr^d\times \M.$$
 It follows from \eqref{e:16} that
$$\G V_n(y, j)\le \G V(y, j)\le -1 \quad \text{for all}\quad (y, j)\in B(0, n)-D.$$
Dynkin's formula implies that \bed \barray \aad \E_{x,
i}V_n\(X\(\sg_D^{(n)}(t)\), \La\(\sg_D^{(n)}(t)\)\)-V_n(x, i)\\
\aad\qquad\qquad =\E_{x, i}\int_0^{\sg_D^{(n)}(t)}\G V_n \(X(s),
\La(s)\)ds\le -\E_{x, i}\sg_D^{(n)}(t).\earray\eed Note that the
function $V_n$ is nonnegative; hence we have $\E_{x,
i}\sg_D^{(n)}(t)\le V_n(x, i)=V(x, i).$ Meanwhile, since the process $(X(s), \La(s))$ is regular, $\be_n\to \infty$ a.s as $n\to \infty$.
As a result, $\sg_D^{(n)}(t)\to \sg_D(t)$ a.s as $n\to \infty$,
where $\sg_D(t) = \min\{\sg_D, t\}$. By virtue of Fatou's lemma, we
obtain $\E_{x,i}\sg_D(t)\le V(x, i)$. Now the argument after
\eqref{e:14} in the proof of \thmref{thm:exit} yields that
\beq{e:17}\E_{x,i}\sg_D \le V (x, i) < \infty.\eeq Since
$(x, i)\in D^c\times \M$ is arbitrary, we conclude that $(X(t), \La(t))$ is
positive recurrent with respect to $U$.
\end{proof}

\begin{thm}\label{thm:rec}
A sufficient condition for the recurrence of $\(X(t), \La(t)\)$ with
respect to $U=D\times \M\subset \rr^d\times \M$
 is the following condition {\rm (C2)} holds:
 \item[\rm (C2)] For
each $i\in \M$, there exists a nonnegative  function
$V(\cdot, i)\in C^2(\rr^d)$ satisfying
\beq{e:18}\barray \aad \G V(x, i)\le 0, \text{ for any }(x, i)\in
D^c\times \M,\\
\aad  \inf\limits_{|x|\ge n,\, i\in \M}V(x, i)\to \infty, \quad n\to \infty.\earray\eeq
\end{thm}

\begin{proof}
Assume that there exists a
nonnegative function $V(\cdot, \cdot)$ satisfying condition (C2) with respect to
$U$. We show that the process $(X(t), \La(t))$ is
recurrent with respect to $U$.

Choose $n_0$ to be a positive integer sufficiently large so that
$D\subset B(0, n_0)$. Fix any $(x, i)\in D^c \times \M$. For any
$t>0$ and $n\ge n_0$, we define $$\sg_D^{(n)}(t)=\min\{\sg_D, t,
\be_n\},$$ where $\be_n$ is the first exit time from $B(0, n)$ and
$\sg_D$ is the first entrance time to $D$. Let $f_n:\rr^d\mapsto \rr$ be a smooth cut-off function that takes values in $[0, 1]$, 1 on $B(0, n)$, and $0$ outside of $B(0, n+1)$.
Denote $M_n:=\inf\limits_{|y|\ge n,\, j\in \M}V(y, j)$. Then $$V_n(\cdot, j):=f_n(\cdot) V(\cdot, j)+\big(1-f_n(\cdot)\big)M_n\in C^2_b(\rr^d), \quad \text{for each} \quad j\in \M.$$ Moreover, $V_n(y, j)\le V(y, j)$ for all $(y, j)$, $V_n(y, j)= V(y, j)$ for $(y, j)\in B(0, n)\times \M$, and $V_n(y, j)\ge M$ for $(y, j)\in B(0, n)^c\times \M$. Detailed computations and \eqref{e:18} yield  that
$$\G V_n(y, j)\le \G V(y, j)\le 0 \quad \text{for all}\quad (y, j)\in B(0, n)-D.$$
Now Dynkin's formula implies that \bed \barray \aad \E_{x,
i}V_n\(X\(\sg_D^{(n)}(t)\), \La\(\sg_D^{(n)}(t)\)\)-V_n(x, i)\\
\aad\qquad\qquad =\E_{x, i}\int_0^{\sg_D^{(n)}(t)}\G V_n \(X(s),
\La(s)\)ds \le 0.\earray\eed
Consequently,
$$\E_{x,
i}V_n\(X\(\sg_D^{(n)}(t)\), \La\(\sg_D^{(n)}(t)\)\) \le V_n(x, i) = V(x, i).$$
By virtue of Fatou's lemma, we
obtain $$\E_{x,
i}V_n\(X\(\sg_D \wedge \beta_n  \), \La\(\(\sg_D \wedge \beta_n  \)\)\) \le V(x, i).$$
Then we have
\begin{equation*} \barray  V(x, i) \ad \ge  \E_{x,
i}\Big[V_n\(X\(\beta_n  \), \La\(\beta_n\)  \)\indi_{ \{  \beta_n<\sg_D  \} }\Big]\\ \ad \ge M_n \P_{x, i}(\beta_n<\sg_D).\earray\end{equation*}
It follows from \eqref{e:18} that as $n\to \infty$,
$$\P_{x, i}(\beta_n<\sg_D) \le \dfrac{V(x, i)}{M_n}\to 0.$$
Note that $\P_{x, i}(\sg_D=\infty)\le \P_{x, i}(\beta_n<\sg_D)$. Hence $\P_{x, i}(\sg_D=\infty)=0$. Since
$(x, i)\in D^c\times \M$ is arbitrary, we conclude that $(X(t), \La(t))$ is recurrent with respect to $U$.
\end{proof}

Now we recall the definition of harmonic functions with respect to process $(X(t), \La(t)$, namely, $\G$-harmonic functions.

\begin{defn}\label{D:har} {\rm
		Let $U=\bigcup_{i=1}^m D_i\times \{i\}$
		with $D_i\subset\R^d$
		being a bounded domain. A bounded and Borel measurable function $f:\R^d\times\M\mapsto\R$ is said to be $\mathcal G$-harmonic in $U$ if for any
		relatively compact open subset $V$ of $U$,
		$$ f(x,i)=\mathbb \E_{x,i}\big[f\(X(\tau_V),\Lambda({\tau_V})\)\big] \quad \text{ for all}\quad (x, i)\in V,$$
		where $\tau_V=\inf\{t\ge 0: \(X(t), \Lambda(t)\) \notin V \}$ is the first exit time of $V$.
}\end{defn}

\begin{defn}{\rm
		The generator $\G$ or the matrix function $Q\cd$
		is said to be \textit{strictly irreducible} on $D\subset \R^d$ if
		for any $i, j\in \M$ and $i\ne j$, there exists $q_{ij}^0>0$ such that $\inf\limits_{x\in D}q_{ij}(x)\ge q_{ij}^0$.	
	}
\end{defn}

In the rest of this section, we assume that
the $Q$-matrix in the operator $\mathcal G$ is \textit{strictly irreducible} on any bounded domain of $\R^d$.
We are in a position to prove that if the process $(X(t), \La(t))$ is recurrent (resp., positive recurrent) with
respect to some cylinder $D\times \M\subset \rr^d\times \M$, then
it is recurrent (resp., positive recurrent) with respect to any
cylinder $E\times \M\subset \rr^d\times \M$.

\begin{lem}\label{lem:re1}
Suppose that the operator $\mathcal G$ is \textit{strictly irreducible} on any bounded domain of $\R^d$.
Let $D\subset \rr^d$ be a bounded domain. Suppose
that \beq{e:21} \P_{x,i}\left( \sg_{ D} < \infty \right) = 1
\quad\text{for any}\quad (x, i)\in D^c\times \M. \eeq Then for any
bounded domain
$E\subset \rr^d$,
 we have \beq{e:22}
\P_{x,i} ( \sg_{ E} <  \infty ) = 1 \quad\text{for any}\quad (x, i)\in
E^c\times \M. \eeq
\end{lem}

\begin{proof}
 Without loss of generality,
we suppose that $\overline{E}\subset D$. Otherwise, let $\wdt D$ be a sufficiently large bounded domain containing both $\overline{D}$ and $\overline{E}$. Then \eqref{e:21} and our arguments hold for $\wdt D$ in place of $D$. It is sufficient to prove \eqref{e:22} for $(x, i)\in \(D \setminus\overline{E}\)\times \M$.
Let $G$ be a bounded domain such that $\overline{D}\subset G$.
Define a sequence of stopping times by \beq{e:23} \tau_0=0, \quad \tau_1=\inf\{t>0:
X(t)\in G^c\}, \eeq and for $n=1, 2, \dots$, \beq{e:24} \barray
\aad \tau_{2n}=\inf\{t> \tau_{2n-1}: X(t)\in D\},\\
\aad \tau_{2n+1}=\inf\{t> \tau_{2n}: X(t)\in G^c\}. \earray \eeq It
follows from \eqref{e:21}, \thmref{thm:exit}, and the strong Markov property that $\tau_n <
 \infty$ $\P_{x,i}$ a.s.  for $n = 1, 2, \dots$ Define $u(y, j)=\P_{y, j}(\sg_{ E}<\tau_G)$ for $(y, j)\in \RR^d\times \M$. Then $u(y, j)=\E_{y, j}\indi_{\{X(\tau_{G \setminus\overline{E}})\in E\}}$. Hence $u(y, j)$ is a $\G$-harmonic function in $\(G \setminus\overline{E}\)\times \M$.
By Remark \ref{R:2.2}(c) and \cite[Theorem 3.4]{CC16},
$u(y, j)> 0$ for all $(y, j)\in (G \setminus\overline{E})\times \M$. Let $H$ be a domain such that $\overline{D}\subset H\subset\overline{H}\subset G$.
Define $$v(y, j)=\E_{y, j}u\(X(\tau_H),\Lambda({\tau_H})\)\quad \text{for}\quad (y, j)\in H\times \M.$$
Again by  Remark \ref{R:2.2}(c) and \cite[Theorem 3.4]{CC16},
$v(y, j)>0$ for $(y, j)\in H\times \M$.
By Remark \ref{R:2.2}(c) and   the Harnack inequality \cite[Theorem 4.7]{CC16},
there is a constant $\delta_1\in (0, 1)$ such that $\inf\limits_{(y, j)\in { D}\times \M}v(y, j)\ge \dl_1$. Since
$u(y, j)$ is a $\G$-harmonic function in $\(G \setminus\overline{E}\)\times \M$, it is a $\G$-harmonic function in $\(H \setminus\overline{E}\)\times \M$. Thus, for any $(y, j)\in D\times \M$,
\bea
u(y, j)\ad =\E_{y, j}\big[u\(X(\tau_H),\Lambda({\tau_H})\)\indi_{\tau_H<\sg_{E}} \big] + \E_{y, j}\big[\indi_{\sg_{E}<\tau_H} \big]\\
\ad \ge v(y, j).\eea
 It follows that \beq{e:25}\inf\limits_{(y, j)\in { D}\times \M}\P_{y, j}(\sg_{ E}<\tau_G)\ge \delta_1.\eeq
 Define \beq{e:26}
A_0=\{X(t)\in {E} \text{ for some }t\in [0, \tau_1)\}, \eeq
and for $n=1, 2, \dots$, \beq{e:27} A_n=\{X(t)\in {E} \text{
for some }t\in [\tau_{2n}, \tau_{2n+1})\}. \eeq Note that the event
$A_0^c$ implies $\tau_G<\sg_{ E}$. Hence we have from
\eqref{e:25} that
$$\P_{x, i}(A_0^c)\le \P_{x, i}(\tau_G<\sg_{ E})\le 1-\delta_1.$$
By the strong Markov property,  induction on $n$ yields
\beq{e:28} \P_{x, i}\(\bigcap\limits_{k=0}^n A_k^c\)\le
(1-\delta_1)^{n+1}. \eeq Thus, we have \bed\barray \aad
\P_{x, i} (\sg_{{E}}=\infty )=\P_{x, i} (X(t)\notin {E}
\text{ for any } t\ge 0)\\
\aad  \qquad\le \lim\limits_{n\to \infty}\P_{x, i}\(\bigcap\limits_{k=0}^n
A_k^c\)\\
\aad \qquad \le  \lim\limits_{n\to \infty}
(1-\delta_1)^{n+1}=0.\earray \eed This proves the lemma.
\end{proof}

\begin{lem}\label{lem:po1}
Suppose that the operator $\mathcal G$ is \textit{strictly irreducible} on any bounded domain of $\R^d$.
Let $D\subset \rr^d$ be a bounded domain
and suppose that $\G$ satisfies condition {\rm (C1)} with respect to $D\times \M$.
Then for any bounded domain
$E\subset \rr^d$,
we have \beq{e:29} \E_{x,i}\sg_E <
\infty\quad\text{for any}\quad (x, i)\in E^c\times \M. \eeq
\end{lem}

\begin{proof} Since  $\G$ satisfies condition {\rm (C1)} with respect to $D\times \M$,
\beq{e:30} \E_{y,j}\sg_D<\infty \quad\text{for
any}\quad (y, j)\in D^c\times \M. \eeq
Using the same reasoning as in the proof for Lemma \ref{lem:re1}, without loss of generality,
we may
assume
 that $\overline{E}\subset D$. It suffices to prove \eqref{e:30} for $(x, i)\in \(D \setminus\overline{E}\)\times \M$.
Let $G$ be a bounded domain such that $\overline{D}\subset G$.

Define a sequence of stopping times
$\{\tau_n\}$
 and events $A_0$,
$A_1$, $A_2$, $\dots$  as in \eqref{e:23}, \eqref{e:24}, \eqref{e:26},
and \eqref{e:27} in the proof of \lemref{lem:re1}. It
follows from \eqref{e:30}, \thmref{thm:exit}, and the strong Markov property that $\tau_n < \infty$ $\P_{x,i}$ a.s.  for $n = 1, 2, \dots$

We now prove that
\beq{e:31}M_2:=\sup\limits_{(y, j)\in D\times \M}\E_{y, j}\tau_2<\infty.\eeq
From Theorem \ref{thm:exit}, we have $M_1:=\sup\limits_{(y, j)\in D\times \M}\E_{y, j}\tau_1 <\infty$. Meanwhile, by \eqref{e:17} in the proof of
Theorem \ref{thm:prec}, we obtain $$\E_{y, j}\sg_D  \le V(y, j)\quad  \text{for}\quad (y, j)\in G^c\times \M.$$
Hence to prove \eqref{e:31}, it suffices to show that
$$\sup\limits_{(y, j)\in D\times \M} \int_{G^c} \sum_{k=1}^m V(z, k)\P_{\tau_1}(y, j, dz, k)<\infty, $$
where
$\P_{\tau_1}(y, j, \cdot, \cdot)$ is the distribution of $(X^{y, j}(\tau_1), \La^{y, j}(\tau_1))$.
Since $V(\cdot, \cdot)$ is bounded on compact sets,  it
is enough
if we can find an open ball $B(0, R)$ with $R$ sufficiently large such that $\{y: d(y, G)<2\}\subset B(0, R)$ and
\beq{e:32}\sup\limits_{(y, j)\in D\times \M} \int_{B(0, R)^c} \sum_{k=1}^m V(z, k)\P_{\tau_1}(y, j, dz, k)<\infty.\eeq
Let a point $x^*\in \partial G$. Then for any $x\in G$ and $z\in B(0, R)^c$,
there is a sequence $\{x_i: i=0, ..., \wdt{n}\}$ such that $x_0=x, x_{\wdt{n}}=x^*$, $|x_i-x_{i-1}|<1/2$ and $x_i\in \overline{G}$  for $i=1, \dots, \wdt{n}$. Since $G$ is bounded, $\wdt{n}$ can be independent of $x$.
By  assumption (A4), we have
$$\wdt\pi_k(x_{i-1}, z-x_{i-1})\le \al_{1/2} \wdt\pi_k( x_{i}, z- x_{i}), \quad i=1, \dots, \wdt{n}, k=1, \dots, m.$$
Thus,
there is a positive constant $K=\ \al_{1/2}^{\wdt{n}}$,  depending only on $G$ such that
$$\wdt\pi_k(x, z-x)\le K \wdt\pi_k(x^*, z-x^*)  \quad \text{for} \quad x\in G, z\in B(0, R)^c, k=1, \dots, m.$$
Let $(y, j)\in D\times \M$ and $A\subset B(0, R)^c$.
  By virtue of  Chen et al. \cite[Proposition 4.3]{CC16},
 $$
\sum_{s\leq t}\indi_{\{X(s-)\in G, X(s)\in A, \La(s)=k\}}-\int_0^t \indi_{\{X(s)\in G, \La(s)=k\}}\int_{A}\wdt\pi_{k}(X(s),z-X(s))dzds,$$
is a $\P_{y, j}$-martingale.
We deduce that
\beq{e:32.1}
\barray
\P_{y, j}\( X(\tau_1\wedge t)\in A , \La(\tau_1\wedge t)=k\) \ad  = \E_{y, j} \bigg[ \sum_{s\le \tau_1\wedge t} \indi_{\{X(s-)\in G, X(s)\in A, \La(s)=k \}}  \bigg]\\
\ad = \E_{y, j} \bigg[ \int_0^{ \tau_1\wedge t} \indi_{\{X(s)\in G, \La(s)=k\}} \int_A \wdt\pi_{k}(X(s),z-X(s))dzds  \bigg]\\
\ad \le K \E_{y, j} \bigg[ \int_0^{ \tau_1\wedge t}   \int_A   \wdt\pi_k(x^*, z-x^*) dzds \bigg]\\
\ad \le K \E_{y, j}\( \tau_1\wedge t \)\mu_k(A),
\earray
\eeq
where $\mu_k$ is a measure on $B(0, R)^c$ with density $\wdt\pi_k(x^*, z-x^*)$. Using assumption (A3) and the fact that $A\cap B(0, R)\subset B(0, R)^c\cap B(0, R)=\emptyset$, we have $\mu_k(A)<\infty$.
Letting $t\to\infty$ and using Fatou's lemma on the left-hand side and the dominated convergence theorem on the right-hand side in \eqref{e:32.1}, we have
\begin{equation*}
\P_{y, j}\( X(\tau_1)\in A, \La(\tau_1)=k \) \le K \E_{y, j} (\tau_1) \mu_k(A)
\le K M_1 \mu_k(A),
\end{equation*}
where $M_1=\sup\limits_{(z, k)\in D\times \M} \E_{z, k} (\tau_1)$. Hence
$$\P_{\tau_1}(y, j, A, k)\le KM_1 \mu_k(A)\quad \text{for all}\quad k\in \M.$$
It follows that
\beq{e:32.2}\barray \disp \int_{B(0, R)^c} \sum_{k=1}^m V(z, k)\P_{\tau_1}(y, j, dz, k) \ad
\le  KM_1\sum_{k=1}^m  \int_{B(0, R)^c}  V(z, k) \wdt\pi_k(x^*, z-x^*) dz \\
\ad = KM_1\sum_{k=1}^m
 \int_{B(-x^*, R)^c}
 V(z+x^*, k) \wdt\pi_k(x^*, z) dz.\earray\eeq
 Choose $R$ sufficiently large  such that $B(0, R)^c-x^*\subset \{z\in \rr^d: |z|>1\}$.
Since $\G$ satisfies condition (C1) with respect to $D\times \M$ and $x^*\notin D$, we have $\G V(x^*, k)<\infty$ for all $k\in \M$,
which leads to the finiteness of the last term in \eqref{e:32.2}.
The desired inequality \eqref{e:32} then follows.

To continue, note that $\tau_{2n}\le \sg_E<\tau_{2n+1}$
implies $\bigcap\limits_{k=0}^{n-1}A_k^c$. Then it follows from
\eqref{e:28} that $$\P_{x, i}\(\tau_{2n}\le \sg_E<\tau_{2n+1}\)\le
\P_{x, i}\(\bigcap\limits_{k=0}^{n-1}A_k^c\)\le (1-\delta_1)^n.$$
Therefore we have \bed{}\barray  \E_{x,
i}\sg_E \ad =\sum\limits_{n=0}^\infty \E_{x, i}\sg_E \indi_{[\tau_{2n}\le
\sg_E<\tau_{2n+2}]}\\
\ad \le \sum\limits_{n=0}^\infty \P_{x, i}\big(\tau_{2n}\le
\sg_E<\tau_{2n+2}\big)\E_{x, i} \tau_{2n+2} \\
\ad \le \sum\limits_{n=0}^\infty \P_{x, i}\big(\tau_{2n}\le
\sg_E<\tau_{2n+1}\big)  \sum\limits_{k=0}^n  \E_{x, i} \(\tau_{2k+2}-\tau_{2k}\)\\
\ad \le \sum\limits_{n=0}^\infty (1-\delta_1)^n (n+1)M_2<\infty.\earray\eed
This completes the proof of the lemma.
\end{proof}

In the next two lemmas, we show
that, under the irreducibility assumption of the $Q$-matrix,
 if the process $(X(t),
\La(t))$ is recurrent (resp., positive recurrent) with respect to
some cylinder $D\times \M\subset \rr^d\times \M$, then it is
recurrent (resp., positive recurrent) with respect to
$D\times \{l\}\subset \rr^d\times \M$ for any $l\in \M$, where
$D\subset \rr^d$ is any bounded domain.

\begin{lem}\label{lem:re2}
Suppose that the operator $\mathcal G$ is \textit{strictly irreducible} on any bounded domain of $\R^d$.
Let $D\subset \rr^d$ be a bounded domain. Suppose that \beq{e:33} \P_{y,j}\(\sg_D < \infty\)
= 1 \quad\text{for any}\quad (y, j)\in D^c\times \M. \eeq Then for
any $(x, i)\in \rr^d\times \M$ and $l\in \M$, we have
\beq{e:34} \P_{x,i}
( \sg_{D\times \{l\}} <\infty) =1.
\eeq
\end{lem}

\begin{proof}
Fix $l\in \M$. It suffices to prove \eqref{e:34} when
$(x, i)\in D\times (\M\setminus \{l\})$ since the process $(X(t),
\La(t))$, starting from $(y, j)\in D^c\times \M$ will reach $D\times
\M$ in finite time $\P_{y, j}$-a.s. by \eqref{e:33}.

Choose $\e>0$ sufficiently small such that $B_0\subset
\overline{B}_0\subset B_1\subset \overline{B_1}\subset D,$ where
\beq{e:35}
B_0=B(x, \e) \quad \hbox{and} \quad
B_1=B(x, 2\e).
\eeq
Define
\beq{e:36} \tau_0=0, \quad \tau_1=\inf\{t> 0: X(t)\in B_0^c\},
\eeq and for $n=1, 2, \dots,$ \beq{e:37}
\tau_{2n}=\inf\{t\ge  \tau_{2n-1}: X(t)\in B_1^c\},\quad
\tau_{2n+1}=\inf\{t\ge  \tau_{2n}: X(t)\in B_0\}. \eeq Note that
\eqref{e:33}, \thmref{thm:exit}, and \lemref{lem:re1} imply
that $\tau_n<\infty$
$\P_{x, i}$-a.s..
Define
$$u(y, j)=\P_{y, j}\(\sg_{\overline{B}_0\times \{l\}}<\tau_{B_1}\), \quad (y, j)\in \rr^d\times \M.$$
Then $u\cd$ is  $\G$-harmonic in
$B_1\times \M \setminus (\overline B_0 \times \{l\})$, and hence in particular in
  $(B_1 \setminus \overline{B}_0)\times \M$.
By  Remark \ref{R:2.2}(c) and \cite[Theorem 3.4]{CC16},
$u>0$ on $(B_1 \setminus \overline{B}_0)\times \M$.
Let $H$ be a domain such that $\overline{B}_0\subset H\subset \overline{H}\subset B_1$.
Define
$$
v(y, j)=\E_{y, j}u\(X(\tau_H),\Lambda({\tau_H})\)\quad \text{for}\quad (y, j)\in H\times \M.
$$
Clearly, $v\geq 0$ is an $\G$-harmonic function on $H\times \M$.
By Remark \ref{R:2.2}(c) and the Harnack inequality \cite[Theorem 4.7]{CC16},
there is a constant $\delta_2\in (0, 1)$ such that $\inf\limits_{(y, j)\in { B_0}\times \M}v(y, j)\ge \dl_2$. For any $(y, j)\in B_0\times \M$, by the definition of $u(y, j)$,
\bea
u(y, j)\ad
\geq \E_{y, j}\big[u\(X(\tau_H),\Lambda({\tau_H})\)\indi_{\tau_H<\sg_{B_0\times \{l\}}} \big]+ \E_{y, j}\big[\indi_{\sg_{B_0\times \{l\}<\tau_H} }\big] \ge  v(y, j).
\eea
Thus, \beq{e:38}\inf\limits_{B_0\times
	\M}\P_{y, j}\(\sg_{\overline{B}_0\times \{l\}}<\tau_{B_1}\)\ge \delta_2.\eeq
Redefine \beq{e:39} A_0=\{\La(t)=l \text{ for some }t\in [\tau_0,
\tau_2)\}, \eeq and for $n=1, 2, \dots$, \beq{e:40} A_n=\{\La(t)=l
\text{ for some }t\in [\tau_{2n+1}, \tau_{2n+2})\}. \eeq The event
$A_0^c$ implies $\sg_{\overline{B}_0\times \{l\}}>\tau_{B_1}.$ Hence
we have by \eqref{e:38} that
$$\P_{x, i}(A_0^c)\le \P_{x, i}\(\sg_{\overline{B}_0\times \{l\}}>\tau_{B_1}\)\le
1-\delta_2,$$ i.e., $\P_{x, i}(A_0^c)\le 1-\delta_2.$ By virtue of the
strong Markov property and induction on $n$, \beq{e:41}\P_{x,
i}\(\bigcap\limits_{k=0}^n A_k^c\)\le (1-\delta_2)^{n+1}.\eeq Thus,
we have \bed\barray \aad
 \P_{x,
i}\((X(t), \La(t))\notin D\times \{l\} \text{ for any }t\ge 0\)\\
\aad \qquad\le \P_{x, i}\((X(t), \La(t))\notin B_1\times \{l\} \text{
for
any }t\ge 0\)\\
\aad \qquad\le \lim\limits_{n\to \infty} \P_{x,
i}\(\bigcap\limits_{k=0}^n
A_k^c\)\\
\aad \qquad\le \lim\limits_{n\to \infty}(1-\delta_2)^{n+1}=0.
 \earray\eed
As a result, $$\P_{x, i}(\sg_{D\times \{ l \}}=\infty)=\P_{x, i}\((X(t),
\La(t))\notin D\times \{l\} \text{ for any }t\ge 0\)=0,$$ i.e.,
$\P_{x, i}(\sg_{D\times \{ l \}}<\infty)=1$, as desired.
\end{proof}

\begin{lem}\label{lem:po2}
Suppose that the operator $\mathcal G$ is \textit{strictly irreducible} on any bounded domain of $\R^d$.
Let $D\subset \rr^d$ be a bounded domain
and suppose that $\G$ satisfies condition {\rm (C1)} with respect to $D\times \M$.
Then for any $(x, i)\in \rr^d\times \M$ and $l\in \M$, we have
\beq{e:42} \E_{x,i}\sg_{D\times \{ l \}} < \infty.\eeq
\end{lem}

\begin{proof} Fix any $l\in \M$. As in \lemref{lem:re2}, it  suffices to prove
\eqref{e:42} when $(x, i)\in D\times \(\M\setminus \{l\}\)$.
Let the balls $B$ and $B_1$, stopping times $\tau_0$, $\tau_1$, \dots,
and events $A_0$, $A_1$, \dots, as in \eqref{e:35}-\eqref{e:37},
\eqref{e:39}, and \eqref{e:40} in the proof of \lemref{lem:re2}.

Observe that $\tau_{2n}\le
\sg_{D,\, l}<\tau_{2n+2}$ implies $\bigcap\limits_{k=0}^{n-1} A_k^c$.
Hence we have by \eqref{e:41} that \bed  \P_{x, i}\(\tau_{2n}\le
\sg_{D\times \{ l \}}<\tau_{2n+2}\)\le \P_{x, i}\(\bigcap\limits_{k=0}^{n-1}
A_k^c\)\le (1-\delta_2)^n.\eed

By virtue of Theorem \ref{thm:exit}, we have
$M_1:=\sup\limits_{(x, i)\in B\times \M}\E_{x, i}\tau_1<\infty$. Similar to the proof of \eqref{e:31}, we also obtain
$M_2:=\sup\limits_{n} \E_{x, i} \( \tau_{2n+3} -\tau_{2n+1}  \)<\infty$. Therefore, we have
\bea  \E_{x,
i}\sg_{D, \,l} \ad =\sum\limits_{n=0}^\infty \E_{x, i}\sg_{D, \,l} \indi_{[\tau_{2n}\le
\sg_{D, \,l}<\tau_{2n+2}]}\\
\ad \le \sum\limits_{n=0}^\infty \P_{x, i}{\(\tau_{2n}\le
\sg_{D, \,l}<\tau_{2n+2}\)}\E_{x, i} \tau_{2n+3} \\
\ad \le \sum\limits_{n=0}^\infty \P_{x, i}{\(\tau_{2n}\le
\sg_{D, \,l}<\tau_{2n+2}\)} \bigg[ \sum\limits_{k=0}^n  \E_{x, i} \(\tau_{2k+3}-\tau_{2k+1}\) +  \tau_1  \bigg]\\
\ad \le \sum\limits_{n=0}^\infty  (1-\delta_1)^n \( nM_2+M_2+ M_1  \) <\infty.\eea
The proof of the lemma is complete.
\end{proof}

By virtue of Lemma \ref{lem:re1} and Lemma \ref{lem:po2}, the
process $(X(t),\La(t))$ is recurrent with respect to some cylinder $D\times \M$ if and only
if it is recurrent with respect to the
product set $D\times \{l\}$ for any $l\in \M$. Also we have proved
that the property of recurrence is
independent of the choice of the bounded domain $D\subset \rr^d$.
Moreover,
similar results also hold for positive recurrence.
We summarize these into the following theorem.

\begin{thm}\label{thm:rec}
Suppose that the operator $\mathcal G$ is \textit{strictly irreducible} on any bounded domain of $\R^d$.
Assume that $D, E\subset \rr^d$ are bounded domains  and $l\in \M$.
The following assertions hold:
\begin{itemize}
\item[{\rm(i)}] The process $(X(t),\La(t))$ is recurrent with respect to $D\times \M$ if and only if it is
recurrent with respect to $E\times
\M$.
The $(X(t),\La(t))$ is recurrent with respect to $E\times
\M$,
if and only if it is
recurrent with respect to $E\times
\{l\}$.
\item[{\rm(ii)}] Suppose that condition {\rm (C1)}
 is satisfied. The process $(X(t), \La(t))$ is positive recurrent with respect to $D\times \M$ if and only if it is
positive recurrent with respect to $E\times\M$, and
 $(X(t),\La(t))$ is recurrent with respect to $E\times
\M$
if and only if it is
recurrent with respect to $E\times
\{l\}$.
\end{itemize}
\end{thm}

\section{Existence of Invariant Probability Measures}\label{sec:inva}
In this section, we establish invariant probability measures of the process
$(X(t), \La(t))$ under the standing assumption (A1)-(A4)
 and condition (C1). That is, the process is positive
recurrent with respect to  $U =E\times \{l\}$,
where $E\subset \rr^d$ is a bounded domain and $l\in \M$ is fixed throughout this
section. By using a method introduced by Khasminskii \cite{K80}, we
characterize the invariant probability measure using an embedded Markov chain.

We work with a bounded domain $D\subset \rr^d$ with  $\overline{E}\subset D$.
Let $\tau_0=0$ and redefine stopping times $\tau_1, \tau_2,
\dots$, inductively as follows. \beq{eq:cycle-again}\barray \aad \tau_{2n+1}=\inf\{t>
\tau_{2n}:
X(t) \notin D\},\\
\aad \tau_{2n+2}=\inf\{t> \tau_{2n+1}: X(t)\in
\overline{E}, \La(t)=l\}.\earray\eeq Now we can divide an arbitrary sample
path of the process $(X(t), \La(t))$ into cycles:
$$[\tau_0, \tau_2), [\tau_2, \tau_4), \dots, [\tau_{2n}, \tau_{2n+2}), \dots$$
Since the process $(X(t), \La(t))$ is  recurrent with respect
to $E\times \{l\}$, by Theorem \ref{thm:exit},
all stopping times $\tau_0<\tau_1<\tau_2< \tau_3<\cdots$
are finite a.s. By virtue of the recurrence of $(X(t), \La(t))$, we may assume without loss of generality that
$(X(0), \La(0))=(x, l)\in \overline{E}\times \{l\}.$
It follows from the strong Markov property of the process
$(X(t),\La(t))$ that the sequence
\beq{eq:disc-MC}(\wdt{X}_n,
l):=\(X(\tau_{2n}), l\), n=0, 1, ... \ \hbox{ is a Markov chain in }\
\overline{E}\times \{l\}.\eeq Therefore, the sequence $(\wdt{X}_n)$ is a Markov chain in
$\overline{E}.$

To establish the existence of a unique stationary distribution of the Markov chain
$(\wdt{X}_n)$, we first recall a result on
invariant probability measures of Doob; see \cite[Section 4.1]{B88}.

\begin{thm}\label{thm:doob}
Suppose that $S$ is a compact metric space, that $\mathcal{B}(S)$
is the Borel $\sg$-algebra, and that $P$ is a linear operator from
$B(S)$ into itself, where $B(S)$ is the Banach space of bounded and
Borel measurable functions from $S$ into $\rr$, such that
\beq{e:44} \left\{\barray
 \aad \|P\phi\| \le \|\phi\|, \quad \text{for all} \quad \phi\in B(S),\\
\aad P \phi=\phi \quad \text{if}\quad \phi=1,\earray \right.\eeq where
$\|\cdot\|$ denotes the sup norm.
 Assume that
there exists a $\delta>0$ such that \beq{e:45}P\indi_F(x)-P\indi_F(y)\le 1-\delta \quad \text{ for }\quad  x, y\in S, \quad
F\in \mathcal{B}(S).\eeq Then there exists a unique probability
measure on $(S, \mathcal{B}(S))$ denoted by $\wdt{\mu}$ such that
$$\left|P^n \phi (x)-\int_S \phi d\wdt{\mu}\right|\le K e^{-\rho n}\|\phi\|,$$
where $\rho=-\ln (1-\delta)$ and $K=2/(1-\delta)$. The measure $\wdt{\mu}$
is the unique invariant probability measure under $P$ on $(S, \mathcal{B}(S))$; that is, the unique probability measure on $S$ such that $$\int_S \phi d
\wdt{\mu}= \int_S P \phi d\wdt{\mu} \quad \text{for}\quad \phi\in B(S).$$
\end{thm}

\begin{lem}\label{lem:5.2} Suppose that the operator $\G$ is strictly irreducible on $D$.
The Markov chain $(\wdt{X}_n)$ has a unique
invariant probability measure $\wdt{\mu}$ on $\overline{E}$ and
for $g\in B(\overline{E})$, we have \beq{e:46}
\int_{\overline{E}}\E_{x, \, l}g(\wdt{X}_1)\wdt{\mu}(dx)=\int_{\overline{E}}g(x)\wdt{\mu}(dx).\eeq
\end{lem}

\begin{proof} For each $\phi\in B(\overline E)$ and $\phi\ge 0$, define \beq{e:47} f(x, i)=\E_{x, i}\phi\( X(\tau_2) \), \quad (x, i)\in \rr^d\times \M.\eeq
We claim that $f\cd$ is $\G$-harmonic in $D\times \M$.
Indeed, let $\overline{f}(x, i) = \E_{x, i} \phi\big(X(\sg_{\overline E\times \{l\}})\big)$ for $(x, i)\in D^c\times \M$.
By the strong Markov property, we obtain $$f(x, i)=\E_{x, i} \overline{f}\( X(\tau_1), \La(\tau_1) \), \quad (x, i)\in \rr^d\times \M.$$
Thus, $f\cd$ is $\G$-harmonic in $D\times \M$.
Next we define $$P \phi(x)= \E_{x, \, l} \phi\(X(\tau_2)\) \quad
\text{ for }\quad \phi\in B(\overline{E}),   x\in \overline E.$$
It follows from \eqref{e:47} that
$P\phi(x)=f(x, l)$ for $x\in \overline E$. Using the Harnack inequality for $\G$-harmonic functions
(see Remark \ref{R:2.2}(c) and  \cite[Theorem 4.7]{CC16}),
there exists a positive constant $\ka_E$ such that
\beq{e:48}P\phi(x)\le \ka_E P\phi(y)\ \text{ for all } \ x, y\in E,   \phi\in B(\overline E).\eeq
Note that $\ka_E$ is independent of $\phi\in B(\overline E)$. It can be seen that
$P: B(\overline{E})\mapsto B(\overline{E})$ is a linear operator satisfying \eqref{e:44}. To proceed, we need to verify \eqref{e:45}.

In the  contrary, assume that \eqref{e:45} were not true. Since
$P \indi_F(x)\in [0,1]$ for all $x\in  \overline{E}$ and $F\in  \mathcal{B}(\overline{E})$,
it follows from our assumption that there is a sequence
$(x_k, y_k, F_k)\subset \overline{E}\times \overline{E}\times \mathcal{B}(\overline{E})$ such that \bed{}
P\indi_{F_k}(x_k)\to 1,\quad P\indi_{F_k}(y_k)\to 0 \quad \text{ as } k\to \infty.\eed
This implies that $$\dfrac{P\indi_{F_k}(y_k)}{P\indi_{F_k}(x_k)}\to 0 \quad
 \text{ as } k\to \infty,$$ contradicting \eqref{e:48}. Thus \eqref{e:45} holds, and the conclusion follows from Theorem \ref{thm:doob}.
\end{proof}

\begin{rem}{\rm Let $\tau$ be an $\mathcal{F}_t$-stopping time with $\E_{x,i}\tau <  \infty$ and
let $f: \rr^d\times \M\mapsto \rr$ be a Borel measurable function. Then
for any $(x, i)\in \rr^d\times \M$,
 \beq{e:49} \E_{x, i}\int_0^\tau f(X(t+s),
\La(t+s))ds=\E_{x, i}\int_0^\tau \E_{X(s), \La(s)}f(X(t), \La(t))ds.
\eeq Indeed, since $\tau$ is an $\mathcal{F}_t$-stopping time,
the function $\indi_{[s<\tau]}$ is $\mathcal{F}_s$-measurable.
Therefore, \bed \barray \aad \E_{x, i}\int_0^\tau f(X(t+s),
\La(t+s))ds=\int_0^\infty \E_{x, i}\E_{x, i}\(\indi_{[s<\tau]}f(X(t+s),
\La(t+s))\,|\,\mathcal{F}_s\)ds \\
\aad \hspace{5.7 cm}=\E_{x, i}\int_0^\infty \indi_{[s<\tau]}\E_{x,
i}\(f(X(t+s),
\La(t+s))\,|\,\mathcal{F}_s\)ds\\
\aad \hspace{5.7 cm}=\E_{x, i}\int_0^\tau \E_{X(s), \La(s)}f(X(t),
\La(t)) ds, \earray \eed as desired.
}\end{rem}

\begin{thm}\label{thm:sta} Suppose that the operator $\G$ is strictly irreducible on $D$ and that condition {\rm (C1)} holds. The process $(X(t), \La(t))$
has an invariant probability measure
$\wdt{\nu}(\cdot,
\cdot)$.
\end{thm}

\begin{proof}
Let $A\in \mathcal{B}(\rr^d)$ and $i\in \M$.
Using the cycle times given in \eqref{eq:cycle-again},
we
define \beq{e:50} \nu(A, i):=\int_{\overline{E}} \wdt{\mu}(dx)\E_{x, \, l}\int_0^{\tau_2}
\indi_{A\times \{i\}}(X(t), \La(t))dt,\eeq where $\int_0^{\tau_2}
\indi_{A\times \{i\}}(X(t), \La(t))dt$ is the time spent by the path of $(X(t), \La(t))$ in
the set $\(A\times \{i\}\)$ in the first cycle $[0, \tau_2)$. Then
$\nu(\cdot, \cdot)$ is a positive measure defined on
$\mathcal{B}\(\rr^d\times \M\)$.

We claim that for any bounded and Borel measurable function $g(\cdot,
\cdot):\rr^d\times \M\mapsto \rr$, \beq{e:51} \sumj \int_{\rr^d} g(x, j)\nu(dx,
j)=\int_{\overline{E}} \wdt{\mu}(dx)\E_{x, \, l}\int_0^{\tau_2} g((X(s),
\La(s))ds,\eeq holds. Indeed, if $g(y, j)=\indi_{A\times\{i\}}(y, j)$
for some $A\in \mathcal{B}(\rr^d)$ and $i\in \M$, then \eqref{e:51} follows directly from \eqref{e:50}. Similarly, we obtain that
\eqref{e:51} holds for $g$ being a simple function
$$g(y, j)=\sum\limits_{k=1}^n c_k \indi_{U_k}(y, j)\quad \text{ where }U_k\in \mathcal{B}(
\rr^d\times \M), c_k\in \rr.$$ Finally, if $g$ is a bounded and
Borel measurable function, \eqref{e:51} follows by approximating $g$ by
simple functions.  Let $f$ be a bounded and Borel measurable function on $\rr^d\times \M$. It follows from \eqref{e:51} for $g(x ,i)=\E_{x,
i} f(X(t), \La(t))$ and \eqref{e:49} that \beq{e:52}\barray \aad \sumi
\int_{\rr^d} \E_{x, i} f(X(t), \La(t))\nu(dx, i)\\
\aad \qquad
=\int_{\overline{E}}\wdt{\mu}(dx)\E_{x, l}\int_0^{\tau_2}\E_{X(s),
\La(s)}f(X(t), \La(t))ds\\
\aad \qquad = \int_{\overline{E}}\wdt{\mu}(dx)\E_{x, l}\int_0^{\tau_2}f(X(t+s),
\La(t+s))ds\\
\aad \qquad =
\int_{\overline{E}}\wdt{\mu}(dx)\E_{x, l}\int_0^{\tau_2}f(X(u),
\La(u))du\\
\aad \qquad \quad + \int_{\overline{E}}\wdt{\mu}(dx)
\E_{x, \, l}\int_{\tau_2}^{t+\tau_2}f(X(u),\La(u))du-\int_{\overline{E}}\wdt{\mu}(dx)
\E_{x, \, l}\int_0^t f(X(u), \La(u))du. \earray\eeq Using
\eqref{e:46} with
$g(x)=\E_{x, l}\int_{\tau_2}^{t+\tau_2}f(X(u),\La(u))du$, we obtain that
\beq{e:53}\barray \aad \int_{{\overline{E}}} \wdt{\mu}(dx)
\E_{x,  l}\int_{\tau_2}^{t+\tau_2}f(X(u),\La(u))du\\
\aad  \qquad =\int_{\overline{E}} \wdt{\mu}(dx)\E_{x, l} \E_{\wdt X_1,l}\int_{0}^{t}f(X(u),\La(u))du\\ \aad  \qquad
=\int_{\overline{E}}\wdt{\mu}(dx)\E_{x, \, l}\int_0^t f(X(u), \La(u))du.
\earray\eeq
As noted in \eqref{eq:disc-MC},
 $(\wdt X_n,l) =(X({2n}),l)$,  sampled from the switching jump diffusion, is a discrete time Markov chain.
By \eqref{e:52}, \eqref{e:53}, and \eqref{e:51}, we have
$$\sumi
\int_{\rr^d} \E_{x, i} f(X(t), \La(t))\nu(dx, i)= \sumi \int_{\rr^d}
f(x, i)\nu(dx, i).$$ In view of the proof of
 \lemref{lem:po2}, we have
\beq{e:5302}
\sup\limits_{x\in \overline{E}}\E_{x,\, l}\tau_2<\infty.
\eeq By \eqref{e:5302} and \eqref{e:50}, for each
$j\in \M$, we obtain
$$\nu(\rr^d, j)\le \int_{\overline{E}}\E_{x, \, l}[\tau_2] \wdt{\mu}(dx)<\infty.$$ Thus, the
normalized measure
$$\wdt{\nu}(A, i)=\dfrac{\nu(A, i)}{\sumj \nu(\rr^d, j)},\quad  (A, i)\in \mathcal{B}(\rr^d)\times \M,$$
defines a desired invariant probability measure. This concludes the proof.
\end{proof}

The following result further characterizes invariant probability measures of $(X(t), \La(t))$.

\begin{prop} Under conditions of Theorem \ref{thm:sta},
	let $\wdt{\nu}$
	be an invariant probability measure of $(X(t), \La(t))$. Then for any bounded domain $D\subset \rr^d$
	 and any $l\in \M$,
	 $\wdt{\nu}(D, l)>0$.	
\end{prop}

\begin{proof} We argue by contradiction. Suppose that $\wdt{\nu}(D, l)=0$. Let $x_0\in D$ and $r\in (0,1)$ such that $B(x_0, 2r)\in D$.
By Remark \ref{R:2.2}(c) and \cite[Proposition 4.1]{CC16},
there exists a constant $c_1>0$ such that
$$\P_{x, i}\( \tau_{B(x, r)} \le c_1 r^2\)\le \dfrac{1}{2} \quad \text{for} \quad (x, i)\in B(x_0, r)\times \M.$$
	Let $t_0=c_1r^2$. It follows that
	\beq{e:54}\barray\P_{x, l}\( \tau_{B(x_0, 2r)} \le  t_0\)\ad
	\le
	\P_{x, l}\( \tau_{B(x,r)} \le  t_0\)
\\
	\ad \le
	 \dfrac{1}{2} , \qquad x\in B(x_0, r).\earray\eeq
	Define $$\tau^* = \inf\{t> 0: \Lambda(t)\ne \Lambda(0)\}, \quad \tau_{B(x_0, 2r)}^*=\tau_{B(x_0, 2r)}\wedge \tau^*.$$
	It can be seen that for any $x\in B(x_0, r)$, there is a number $c_2>0$ such that $\P_{x, l}(\tau^*>t_0)\ge \exp(-c_2t_0)$.
	We deduce that  \bea
	\P_{x, l}(\tau^*\le t_0)\ad =
	1-\P_{x, l}(\tau^*>t_0)\\
	\ad  \le  1 - \exp(-c_2 t_0)\\
	\ad \le c_2 t_0.
	\eea
	Take a smaller $r$ if needed to have $c_2c_1r^2\le 1/8$. Then we obtain $\P_{x, l}(\tau^*\le t_0)\le 1/8$ for all $x\in B(x_0, r)$.
	Together with \eqref{e:54}, we arrive at
	\beq{e:55}\barray
	\P_{x,l}\(\tau_{B(x_0, 2r)}^*\ge  t_0\)\ad \ge 1-\P_{x,l}\(\tau_{B(x_0, 2r)}\le t_0\)-\P_{x,l}(\tau^*\le  t_0)\\
	\ad \ge \dfrac{1}{4}, \qquad x\in B(x_0, r).
	\earray\eeq
	Let $\K\subset \rr^d$ be a compact
	set satisfying $\sumi\wdt{\nu}(\K, i)>1/2$. By Lemma \ref{lem:po2} and its proof,
there is a constant $c_3$ such that $$\E_{x, i}\sg_{B(x_0, r), \, l}\le c_3 \quad  \text{for any} \quad (x, i)\in \K\times \M.$$ Then for $t_1=2c_3$,
	\beq{e:56}
	\inf\limits_{(x, i)\in \K\times \M }\P_{x, i}
	\(\sg_{B(x_0, r), \, l} > t_1\)\le \dfrac{c_3}{t_1}=\dfrac{1}{2}.
	\eeq
	Since
	$\wdt{\nu}$
	is an invariant probability measure of $(X(t), \La(t))$,
	\begin{equation*}\barray \disp\sumi
	\int_{\rr^d} \E_{x, i} \indi_{D\times \{l\}}(X(t), \La(t))\wdt{\nu}(dx, i)\ad = \sumi \int_{\rr^d}
	\indi_{D\times \{l\}}(X(t), \La(t))\wdt{\nu}(dx, i)\\
	\ad =\wdt{\nu}(D, l).
	\earray
	\end{equation*}
	Using \eqref{e:55} and  \eqref{e:56}, we obtain
	\bea
		\wdt{\nu}(D, l)\ad =\dfrac{1}{t_0+t_1}\sumi\int_0^{t_0+t_1}dt \int_{\rr^d} \E_{x, i} \indi_{D\times \{l\}}(X(t), \La(t))\wdt{\nu}(dx, i)\\
\ad \ge \dfrac{1}{t_0+t_1}\sumi\int_0^{t_0+t_1}dt \int_{\K} \E_{x, i} \indi_{B(x_0, 2r)\times \{l\}}(X(t), \La(t))\wdt{\nu}(dx, i)\\
\ad = \dfrac{1}{t_0+t_1}\sumi\int_{\K}\wdt{\nu}(dx, i)\E_{x, i} \int_0^{t_0+t_1}   \indi_{B(x_0, 2r)\times \{l\}}(X(t), \La(t))dt\\
\ad \ge \dfrac{1}{t_0+t_1}\sumi\int_{\K}\wdt{\nu}(dx, i)\E_{x, i}\bigg[
\indi_{\{ \sg_{B(x_0, r), \, l} \le t_1 \}}
\E_{X(\sg_{B(x_0, r), \, l}),\, l}\Big[
\indi_{\{ \tau_{B(x_0, 2r)}^* \ge t_0\}}\\
\ad \hspace{6.5 cm}
 \int_{\tau_{B(x_0, r)}^*}^{t_0+t_1}   \indi_{B(x_0, 2r)\times \{l\}}(X(t), \La(t))dt\Big]
 \bigg]\\
 \ad \ge \dfrac{\sumi \wdt{\nu}(\K, i)}{t_0+t_1}\inf\limits_{(x, i)\in \K\times \M }\P_{x, i}
 \(\sg_{B(x_0, r), \, l}\le t_1\)\inf\limits_{x\in B(x_0, r) }\P_{x,l}\(\tau_{B(x_0, 2r)}^*\ge  t_0\)t_0\\
 \ad \ge \dfrac{t_0}{16(t_0+t_1)},
\eea
contradicting to $\wdt{\nu}(D, l)=0$. This proves the proposition.
\end{proof}

\section{Examples}\label{sec:exm}

 This section is devoted to several examples. In the first two examples, we give sufficient conditions
for the existence of Lyapunov functions needed in the theorems above and thus the main
results of this paper apply. In the third example, we point out the effect of spatial jumping component to the  asymptotic properties of regime-switching  jump diffusions.
 In the last example, we mention the implication of the main results to controlled two-time-scale system.

\begin{exm}\label{E:5.1}
	{\rm  Assume that assumptions (A1)-(A4) hold and  that the operator $\mathcal G$ is strictly irreducible. 		
We further assume that there exists a constant $\dl\in (0, 2)$ such that for any $i\in \M$ and $N\in (0, \infty)$,
\beq{e:57.0}\dfrac{b'(x, i)x}{|x|^{2-\dl}} + N\int_{|z|>1}  |z|^{\dl} \wdt\pi_i (x, z)dz \to -\infty \quad \text{as}\quad |x| \to \infty.\eeq	
	Then condition (C1)
	is satisfied and
	the process $\(X(t), \La(t)\)$ is positive recurrent.	To show this, let  $f\in C^2(\rr^d
	)$ be a nonnegative function such that $f(x)=|x|^\dl$ for $|x|\ge 1$. Define $V(x, i)=f(x)$ for $(x, i)\in \rr^d\times \M$.
Moreover, the gradient and
Hessian matrix are given by
$$\nabla f(x) = \dfrac{\dl x}{|x|^{2-\dl}}, \quad \nabla^2 f(x)=\dfrac{\dl I_d}{|x|^{2-\dl}}-\dfrac{\dl(2-\dl)xx'}{|x|^{4-\dl}}, \quad |x|\ge 1,$$
where $I_d$ is the $d\times d$ identity matrix.
For $|x|\ge 1$, we have
\beq{e:57.2}\barray
\G V(x, i) \ad =\dfrac{1}{2} \tr\Big[a(x, i)\nabla^2f(x) \Big] +\dfrac{\dl b'(x, i)x}{|x|^{2-\dl}}\\
\ad \qquad +\int_{|z|\le 1} \big[f(x+z)-f(x)-\nabla f(x)  \cdot z \big]\wdt\pi_i(x,z)(dz)\\
\ad \qquad + \int_{|z|> 1} \big[ f(x+z)-f(x) \big]\wdt \pi_i(x,z)(dz).
\earray
\eeq
Since $\nabla^2 f(\cdot)$ and $a(\cdot, i)$
are bounded, we have
\beq{e:57}\dfrac{1}{2}\tr\big[a(x, i)\nabla^2 f(x) \big] +\int_{|z|\le 1} \big[f(x+z)-f(x)-\nabla f(x)  \cdot z \big]\wdt\pi_i(x,z)(dz)\le c_1,\quad x\in \rr^d,\eeq
for some constant $c_1$. Also for some constant $c_2$, we have
\beq{e:58}\int_{|z|> 1} \big[ f(x+z)-f(x) \big]\wdt\pi(x,z)(dz)\le c_2 \int_{|z|>1}  |z|^{\dl} \wdt\pi_i (x, z)dz \quad \text{for all}\quad |x|\ge 1.\eeq
It follows from \eqref{e:57.2}, \eqref{e:57}, and \eqref{e:58} that
$$\G V(x, i)\le c_1+\dfrac{\dl b'(x, i)x}{|x|^{2-\dl}}+ c_2 \int_{|z|>1}  |z|^{\dl} \wdt\pi_i (x, z)dz.$$
Using \eqref{e:57.0},
 there exists a constant $c_3$ such that
$\L V(x, i)\le -1$ for
 $|x|\ge c_3$ and $i\in \M$. By virtue of Theorem \ref{thm:prec}, the process $\(X(t), \La
(t)\)$ is positive recurrent.	
		
	}
\end{exm}

\begin{exm}\label{E:5.2}
	{\rm Suppose that assumptions (A1)-(A4) hold and $d=1$. We assume that the operator $\mathcal G$ is strictly irreducible and $\int_{|z|\ge 1} |z|\wdt \pi_i(x, z)dz<\infty$ for any $(x, i)$.
		Define
		$$A(x, i)=b(x, i) + \int_{1<|z|\le |x|}z \wdt \pi_i(x, z)dz, \quad B(x, i) =  \int_{|z|> |x|}|z| \wdt \pi_i(x, z)dz.$$
		We claim that the process $\(X(t, \La(t)\)$ is positive recurrent if
		\beq{e:59}\limsup\limits_{|x|\to \infty} \big[A(x, i)\sgn(x) + B(x, i)\big]<0.
	\eeq
		Indeed, let $f\in C^2(\R)$ be a nonnegative function such that $f(x)=|x|$ for $|x|>1$ and $f(x)\le |x|$ for $|x|\le 1$. Define $V(x, i)=f(x)$ for $(x, i)\in \rr\times \M$. For sufficiently large $|x|$, we claim that
		\beq{e:60}\barray
	\aad 	I_1 := \int_{|z|\le 1} \big[f(x+z)-f(x)-\ f'(x)   z \big]\wdt\pi_i(x,z)(dz)=0,\\
	\aad I_2:=\int_{1<|z|\le |x|} \big[f(x+z)-f(x) \big]\wdt\pi_i(x,z)(dz)\le \sgn(x) \int_{1<|z|\le |x|}z \wdt \pi_i(x, z)dz,\\
	\aad I_3:=\int_{|z|> |x|} \big[f(x+z)-f(x) \big]\wdt\pi_i(x,z)(dz)\le  \int_{|z|> |x|}|z| \wdt \pi_i(x, z)dz.
		\earray\eeq
		Indeed, suppose $x>0$ and $|x|$ is large. It is clear that $I_1=0$. Note that for $z\in [-x, -x+1]$, $f(x+z)\le x+z$ and for $z\in [-x+1, -1]$, $f(x+z)=x+z$. It follows that
		\bea
		I_2\ad =\int_{1}^x  \big[f(x+z)-f(x) \big]\wdt\pi_i(x,z)(dz)+ \int_{-x}^{-1} \big[f(x+z)-f(x) \big]\wdt\pi_i(x,z)(dz)\\
		\ad =  \int_{1}^x z\wdt\pi_i(x,z)(dz)+ \int_{-x}^{-x+1} \big[f(x+z)-f(x) \big]\wdt\pi_i(x,z)(dz) \\
		\ad \qquad + \int_{-x+1}^{-1} \big[f(x+z)-f(x) \big]\wdt\pi_i(x,z)(dz) \\
		\ad = \int_{1}^x z\wdt\pi_i(x,z)(dz)+ \int_{-x}^{-x+1} z\wdt\pi_i(x,z)(dz) + \int_{-x+1}^{-1} z\wdt\pi_i(x,z)(dz)\\
		\ad = \int_{1<|z|\le |x|}z \wdt \pi_i(x, z)dz.
		\eea
		Note also that for $z\in [-x-1, -x]$, $f(x+z)\le -x-z$ and for $z\in [-\infty, -x-1]$, $f(x+z)=-x-z$. It follows that
		\bea
		I_3\ad =
		\int_{1}^{\infty}  \big[f(x+z)-f(x) \big]\wdt\pi_i(x,z)(dz)+\int_{-\infty}^{-1}  \big[f(x+z)-f(x) \big]\wdt\pi_i(x,z)(dz)\\
		\ad =\int_{1}^{\infty}  z\wdt\pi_i(x,z)(dz)+\int_{-\infty}^{-x-1}  \big[f(x+z)-f(x) \big]\wdt\pi_i(x,z)(dz)\\
		\ad \qquad + \int_{-x-1}^{-x} \big[f(x+z)-f(x) \big]\wdt\pi_i(x,z)(dz)\\
		\ad =  \int_{1}^{\infty}  z\wdt\pi_i(x,z)(dz) +  \int_{-\infty}^{-x-1} (-z-2x)\wdt\pi_i(x,z)(dz)+ \int_{-x-1}^{-x} (-z-2x)\wdt\pi_i(x,z)(dz)\\
		\ad \le \int_{|z|> |x|}|z| \wdt \pi_i(x, z)dz.
		\eea
		Thus, \eqref{e:60} is proved for $x>0$. Similar argument leads to \eqref{e:60} for $x<0$.
		Detailed computations give us that
	$
\G V(x, i)=A(x, i)\sgn(x) + B(x, i)$ for sufficiently large $|x|$ and $i\in \M$. By \eqref{e:59}, condition (C1) is satisfied. Thus, $\(X(t), \La(t)\)$ is positive recurrent. 		
	}
\end{exm}

\begin{exm}\label{E:5.3}
	{\rm
		To illustrate the effect of the jump component on asymptotic properties of regime-switching
	 		jump diffusions,
		consider a Markov process
		$\(X^0(t), \La^0(t)\)$
		associated with the operator
		\bea
		\mathcal{G}^0 f(x,i) \ad=\dfrac{1}{2}a(x,i){\frac{\partial^2f(x,i)}{\partial x^2}+b(x,i)\frac{\partial f(x,i)}{\partial x}} + Q(x)f(x, \cdot)(i), \quad   (x, i)\in \R\times \M,\eea
		where $Q\cd$ is strictly irreducible and assumptions (A1)-(A2) are satisfied.
		Then
		$\(X^0(t), \La^0(t)\)$ is the unique solution of the stochastic differential equations with regime-switching given by
		\bea
	\aad 	d X^0(t) = b\(X^0(t), \La
		(t)\) dt + \sqrt{2 a(X^0(t), \Lambda(t))} dw(t),\\
		\aad \P\Big(\La^0(t+\Delta t)=j | \La^0(t) = i, X^0(s), \La^0(s), s\le t \Big)= q_{ij} \(X^0(t)\)\Delta t + o(\Delta t).
		\eea
		Suppose $b(x, i )=1$ for $x>1$ and
		$b(x, i )=-1$ for $x<-1$.
		Using \cite[Theorem 3.23]{YZ10}, we can check that there exits $(y, l)\in \RR^d\times \M$ and a bounded domain $D\subset \RR^d$ such that $\P_{y, l}(\sg_D=\infty)>0$. Hence, $\(X^0(t), \La^0(t)\)$ is not recurrent (i.e., it is  transient). It follows that $\P_{y, l}\Big( |X^0(t)|\to \infty \ \text{ as }\ t\to \infty \Big)=1$ for all $(y, l)\in \RR\times \M$; see
		 \cite[Chapter 3]{YZ10}.
		
		Next we consider 	 a Markov process
		$\(X(t), \La(t)\)$
		associated with the operator
		$$\G f(x, i)=\G^0 f(x, i) + \int_{\R^d}  \left( f(x+z,i)-f(x,i)-\nabla f(x,i)\cdot z \indi_{\{|z|<1\}}\right)
		\wdt \pi_i(x, z)dz,$$
		where $\wdt \pi_i(x, z)$ satisfies assumptions (A3)-(A4) and \eqref{e:59}. Ignoring the drift term $b(x, i)$, condition \eqref{e:59} means that for $x<0$ and $|x|$ is sufficiently large, the
		intensity of large jumps to the right (i.e., $z>0$) is much stronger than that to the left (i.e., $z<0$); while for $x>0$ and $|x|$ is sufficiently large, the
		intensity of large jumps to the left  is much stronger than that to the right. By our preceding example, $\(X(t), \La(t)\)$ is positive recurrent. Thus, the jump component has certain stabilization effect on regime-switching Markov processes.			
		}
\end{exm}

\begin{exm}\label{E:5.4}
{\rm In this example, we consider a controlled two-time-scale system.
Because it is meant to be for demonstration purpose only,
 we will be very brief. Suppose that we have a fast varying switching jump diffusion $Z^\e(t)=(X^\e(t),\Lambda^\e(t))$ with the generator
 $\LL^\e$
 given by \eqref{E:1} (with the corresponding coefficients indexed by $\e$) such that $a^\e(x,i)= \wdh a(x,i)/\e$, $b^\e(x,i) = \wdh b(x,i) /\e$, $\pi^\e_i(x,i)= \wdh \pi(x,i)/\e$, and $Q^\e(x) = \wdh Q(x)/\e$, and that $\wdh b(\cdot,\cdot)$,
 $\wdh a(x,i)$,
 $\wdh \pi_i(\cdot,\cdot)$, and $\wdh Q\cd$ satisfy  conditions (A1)-(A4).
Hence  conditions (A1)-(A4) are all satisfied for this operator. Moreover, we suppose that $\wdh Q\cd$ is strictly irreducible on some bounded domain of $\R^d$.

Consider a second controlled process $Y^\e\cd$
  depending on $X^\e\cd$ with operator given by
\bea L_\zeta f(y) \ad =  [\nabla f(y)]' \int_U b_1(y,\zeta,c) m_t(dc) +{1\over 2} \tr[ \sigma_1(y,\zeta) \sigma'_1(y,\zeta) \nabla^2f(y)],\eea
where $L_\zeta$ denotes that the operator depends on $\zeta=(x,i)$ as a parameter,
$m_t\cd$ is a relaxed control representation
(see \cite[Chapter 3.2]{Ku90} for the notation), and $U$ is a compact subset of $\R^{d_1}$ (for some positive integer $d_1$) representing the control set. Suppose that
 for each $i\in \M$ and each $x$ and each $c$,
 $b_1(\cdot,\zeta,c)$ and $\sigma_1(\cdot,\zeta)$ are $C^2$ and satisfy the linear growth and Lipschitz condition. Both $X^\e\cd$ and $Y^\e\cd$ take values in $\R^d$. Note that $(X^\e\cd,\La^\e\cd)$ is fast varying, whereas $Y^\e\cd$ is slowly changing.
 Although it varies fast, $(X^\e\cd,\La^\e\cd)$ does not blow up. With a time scale change,
 consider $X_0^\e(t)=X^\e(\e t)$ and $\La^\e_0(t)=\La^\e(\e t)$.
 As $\e\to 0$,
 $(X^\e_0\cd, \La_0^\e\cd)$ is running on an infinite horizon $[0,\infty)$.  Using the results of this paper, it can be shown that as $\e\to 0$,
 the fast process  $(X^\e_0\cd,\La^\e_0\cd)$ has an invariant measure $\wdh \nu(\cdot,\cdot)$.

 Consider a control problem in a finite horizon $[0,T_1]$ with $T_1<\infty$. We wish to minimize an objective
 function
 $$J^\e_\zeta(y,m^\e\cd)= \E \int^{T_1}_0 \int_U G(Y^\e(t), X^\e(t), \La^\e(t), c) m^\e_t(dc),$$
 where $G\cd$ is a running cost function. In the above, we index the relaxed control $m$ by $\e$ to indicate that it is a feedback control.
 Using the approach in \cite{Ku90}, we can show that
 $(Y^\e\cd,m^\e\cd)$ converges weakly to $(Y\cd, m\cd)$, which is a controlled diffusion process with operator given by
 \beq{lim-op} \lbar L  f(y)  =
 [\nabla f(y)]'  \int_U\lbar b_1(y, c) m_t(dc) +  {1\over 2} \tr [\lbar a_1(y) \nabla^2 f(y)], \eeq
 where \beq{bar-quantity}\barray
  \lbar b_1(y,c)\ad= \sum_{j\in \M}\int b_1(y,x,j,c) \wdh \nu (dx,j)
 ,\\ \lbar a_1(y)\ad= \sum_{j\in \M}\int \sigma_1(y,x,j) \sigma'_1 (y,x,j) \wdh\nu(dx,j).\earray\eeq
 The limit  cost function (the limit of $J^\e_\zeta$) then becomes
 $$\lbar J(y,m\cd)= \E \sum_{j\in \M} \int^{T_1}_0 \int_U \int G(Y(t), x, j, c) m_t(dc) \wdh \nu(dx,j).$$
Then we can find the optimal control or near optimal control of the averaged system \eqref{lim-op} and use it in the original system for approximation to get near optimality. Note that in \cite{Ku90}, the controlled systems are given by using the corresponding differential equations, whereas here we present the system using the associated operators. They are in fact, equivalent. The key idea in \cite{Ku90} is to use controlled martingale to obtain the optimality. Here using the operators, the controlled martingales can be easily setup.
In contrast to \cite{Ku90}, the jump process in this paper can have $\sigma$-finite jump measure. The established invariant measure enables us to extend the results in \cite{Ku90} to jumps living in a non-compact set.
Our results thus paved a way for further study on various control and optimization problems involving ergodicity.
}\end{exm}

\section{Further Remarks}\label{sec:rem}
This paper focused on recurrence and ergodicity of a class of switching jump diffusion processes.
Criteria for recurrence and positive recurrence were derived, and existence of invariant measure was obtained.
The results obtained here will help future study
 in controlled dynamic systems, in particular, long-run average cost per unit time problems for
controlled switching jump diffusions. Such study will have great impact on various applications involving optimal controls in an infinite horizon.

\end{document}